\documentclass[12pt]{amsart}

\usepackage[utf8]{inputenc}
\usepackage{amsfonts}
\usepackage{amsthm}
\usepackage{amssymb}
\usepackage{amsmath}
\usepackage{amscd}
\usepackage{latexsym,dsfont}
\usepackage{bbm}
\usepackage{mathrsfs}
\usepackage{xcolor}

\usepackage{times}
\usepackage{microtype}
\usepackage{setspace}
\usepackage[margin=1.2in]{geometry}

\usepackage{cite}

\usepackage[colorlinks=true, pdfstartview=FitV, linkcolor=blue,
citecolor=blue, urlcolor=blue]{hyperref}

\newtheorem{theorem}{Theorem}[section]
\newtheorem{lemma}[theorem]{Lemma}

\newtheorem{defn}[theorem]{Definition}

\newtheorem{hyp}[theorem]{Hypothesis}

\theoremstyle{definition}

\theoremstyle{remark}
\newtheorem{remark}[theorem]{Remark}

\numberwithin{equation}{section}
 \allowdisplaybreaks

\def\XXint#1#2#3{{\setbox0=\hbox{$#1{#2#3}{\int}$}
     \vcenter{\hbox{$#2#3$}}\kern-.5\wd0}}

\DeclareMathOperator{\Div}{div}
\DeclareMathOperator{\supp}{supp}

\DeclareMathOperator{\grad}{\nabla}

\DeclareMathOperator{\lip}{\mathrm{Lip}}

\newcommand{\N}{\mathbb N}

\newcommand{\R}{\mathbb{R}}

\newcommand{\bH}{\mathbf H}
\newcommand{\bG}{\mathbf G}
\newcommand{\bR}{\mathbf R}
\newcommand{\bS}{\mathbf S}
\newcommand{\bT}{\mathbf T}

\newcommand{\vecg}{\mathbf g}

\title[Weak maximum principle]
{The weak maximum principle for solutions of degenerate elliptic equations with lower order terms }

\author[Cruz-Uribe \& Rodney] {  David Cruz-Uribe OFS and   Scott Rodney}

\address{David Cruz-Uribe, OFS \\
Dept. of Mathematics \\
University of Alabama \\
 Tuscaloosa, AL 35487, USA}

\email{dcruzuribe@ua.edu}

\address{Scott Rodney\\
Dept. of Mathematics, Physics and Geology \\ 
Cape Breton University \\
Sydney, NS B1Y3V3, CA} 

\email{scott\_rodney@cbu.ca}

\thanks{The first author is partially supported by a Simons Foundation
  Travel Support for Mathematicians Grant and by NSF Grant DMS-2349550.
   The second author  is partially supported by an NSERC development grant.  This project is supported by  TUBITAK, the Scientific and Technological Research Council of T\"urkiye through a 2501 Joint Research Program grant 223N112; Yusuf Zeren, Y\i ld\i z Technical University,T\"urkiye, is PI.  The first author would also like to thank the organizers of the XVth International Conference of the Georgian Mathematical Union for the invitation to speak.  This paper builds on results presented at the conference and is in response to a question asked after the talk.}


\makeatletter
\def\l@subsection{\@tocline{2}{0pt}{4pc}{5pc}{}}
\makeatother

\begin{document}
\begin{abstract}
 We prove a weak maximum principle for subsolutions of 
 a degenerate, linear, second order elliptic operator with lower order terms, building on the existence results recently proved by the authors and \c{C}etin, Dal and Zeren.
\end{abstract}

\maketitle


\section{Introduction}
The purpose of this note is to prove a weak maximum principle for subsolutions of a Dirichlet problem for a degenerate elliptic operator $L$ defined on a bounded domain $\Omega \subset \R^n$.   Recently, the authors and \c{C}etin, Dal and Zeren~\cite{Tubitak1} proved the existence and uniqueness of solutions to the nonhomogeneous Dirichlet problem 
\begin{equation} \label{eqn:dirichlet} 
\begin{cases}
Lu = f + \bT' \vecg, & x \in \Omega, \\
 u = h, & x \in \partial \Omega,
 \end{cases}
\end{equation}
where $\Omega$ is a bounded, open set in $\R^n$, and $L$ is a degenerate, second order elliptic operator with lower order terms:
\[  Lu = -v^{-1}\Div(Q\grad u)  + \bH\cdot \bR u + \bS'\bG u + Fu.\]
We will give precise definitions of the coefficients below; here we note that $Q$ is an $n\times n$ self-adjoint, positive semidefinite, measurable matrix function, and $v$ is a nonnegative measurable function.  We assume that $Q$ satisfies a degenerate ellipticity condition.  As part of their proof, the authors proved~\cite[Theorem~5.1]{Tubitak1} that the only solution of the homogeneous Dirichlet problem $Lu=0$ with $h=0$ is the $0$ function.  

Our goal is to show  that their arguments can be generalized to prove the stronger result that sub and supersolutions satisfy a weak maximum principle: that is, we show that if $Lu\leq 0$ on $\Omega$, then
\[ \sup_\Omega u(x) \leq \sup_{\partial \Omega} u(x),  \]
and that if $Lu\geq 0$ on $\Omega$, then
\[ \inf_\Omega u(x) \geq \inf_{\partial \Omega} u(x). \]
If $Lu = -\Delta u$, or more generally, if $Lu= -\Div(Q\grad u)$, where $Q$ is a uniformly elliptic matrix, then this result is classical: see~\cite[Theorem~8.1]{MR1814364}.  Fabes, Kenig and Serapioni~\cite[Theorem~2.2.2]{MR643158} proved a weak maximum principle in the case of a degenerate equation where $Q=w\tilde{Q}$ for $\tilde{Q}$ uniformly elliptic and $w$ satisfying the Muckenhoupt $A_2$ condition.  Other, stronger maximum principles have been proved for a variety of degenerate operators with various assumptions: see, for instance,~\cite{MR2015404,MR4224718,CUMR2025}.  We note that these results require some variation of Moser or De Giorgi iteration, whereas our proof is much more elementary.

The remainder of this paper is organized as follows.  In Section~\ref{section:spaces} we define the spaces wherein we define our degenerate weak solutions and in Section~\ref{section:pde} we precisely define the operator $L$.  Finally, in Section~\ref{section:main-thm} we state and prove our main result.  Throughout this paper we are drawing heavily on the groundwork done in~\cite{Tubitak1}, and we refer the reader to this paper for complete details.

\section{Solution spaces for degenerate elliptic PDEs}
\label{section:spaces}
In this section we define the matrix weighted Sobolev space $QH^{1}_0(v,\Omega)$, the solution space for the Dirichlet problem \eqref{eqn:dirichlet}.  For the rest of this section, $v\in L^1(\Omega)$ will denote a non-negative measurable weight function defined on $\Omega$, and $Q:\Omega\rightarrow S_n$ a measurable, non-negative definite, symmetric matrix valued function on $\Omega$.  While me make no assumption on the non-negative lower eigenvalue of $Q$, we do assume an upper ellipticity condition for all $\xi\in\mathbb{R}^n$:
\begin{align}\label{eqn:upperellipticity} \langle Q\xi,\xi\rangle = \left|\sqrt{Q(x)} ~\xi\right|^2 \leq cv(x)|\xi|^2
\end{align}
where $c$ is independent of $x,\xi$.  This condition ensures that the pointwise operator norm of $Q$ is integrable:
\[ \displaystyle \sup_{\xi \in\mathbb{R}^n;~|\xi|=1}\left|\sqrt{Q(x)}\xi\right| \in L^1(\Omega).\]
Associated to $v$ and $Q$ is the subspace of the locally-Lipschitz functions on $\Omega$, 
\[ Q\textrm{Lip}_{loc}(v,\Omega) = \{ \varphi\in Lip_{loc}(\Omega)~:~ \|\varphi\|_{L^2(v,\Omega)}+\|\nabla v\|_{QL^2(\Omega)}<\infty\}\]
where 
\[QL^2(v,\Omega)=\{{\bf g}:\Omega\rightarrow \mathbb{R}^n~:~ \|\nabla\varphi\|_{QL^2(\Omega)}^2=\int_\Omega|\sqrt{Q}~{\bf g}|^2~dx<\infty\}.\]
Since $v\in L^1(\Omega)$ and \eqref{eqn:upperellipticity} holds it is easy to see that the collection $Lip_0(\Omega)$ of Lipschitz functions with compact support in $\Omega$ is a subset of $QLip(v,\Omega)$.  We are now in a position to define the matrix weighted Sobolev space $QH^1_0(v,\Omega)$.
\begin{defn} $QH^1_0(v,\Omega)$ is the completion of $Lip_0(v,\Omega)$ with respect to the norm
\[\| \varphi \|_{QH^1(v,\Omega)} = \|\varphi\|_{L^2(v,\Omega)} + \|\nabla \varphi\|_{QL^2(\Omega)}.\]
\end{defn}
By \cite[Lemma 2.2]{MR3846744}, $QL^2(\Omega)$ is a Banach space with respect to the equivalence ${\bf h} = {\bf g}$ if and only if $\|{\bf h}-{\bf g}\|_{QL^2(\Omega)} = 0$.  This allows for a useful pairwise interpretation of $QH^1_0(v,\Omega)$.  Indeed, given an equivalence class associated to a particular Cauchy sequence of $QLip(v,\Omega)$ functions $\{\varphi_k\}$, it has limit $(u,{\bf g})$ in $L^2(v,\Omega)\oplus QL^2(\Omega)$.  Since this limit is unique to the equivalence class of $\{\varphi_k\}$, $QH^1_0(v,\Omega)$ is isometrically equivalent to a closed subspace of $L^2(v,\Omega)\oplus QL^2(\Omega)$.  Thus, every element of $QH^1_0(v,\Omega)$ is identified as a pair $(u,{\bf g})$ with the property that there is a sequence $\{\varphi_k\}$ of $Lip_0(\Omega)$ functions Cauchy in $\|\cdot \|_{QH^1(v,\Omega)}$ so that
\begin{enumerate}
\item $\varphi_k \rightarrow u$ in $L^2(v,\Omega)$, and
\item $\nabla \varphi_k \rightarrow {\bf g}$ in $QL^2(\Omega)$.
\end{enumerate}
\begin{remark} The function ${\bf g}$ above is thought of as a ``degenerate weak gradient" of $u$ but we note that, due to the possible degeneracy of $Q(x)$, $g$ may not be uniquely determined by $u$.  Nevertheless, if we are given a particular element $(u,{\bf g})$ of $QH^1_0(v,\Omega)$, we will often refer to its derivative as ${\bf g} = \nabla u$ and the pair as $(u,\nabla u)$.
\end{remark}

As a final note in the construction of this space, $QH^1_0(v,\Omega)$ is a Hilbert space with respect to the inner product 
\[\langle(u,{\bf g}),(z,{\bf h})\rangle_Q = \int_\Omega uz~vdx + \int_\Omega \langle Q{\bf h},{\bf g\rangle}~dx   \]
since $\|(u,{\bf g})\|_{QH^1(v,\Omega)}^2 \approx \langle (u,{\bf g}),(u,{\bf g})\rangle_Q$ for any $(u,{\bf g})\in QH^1_0(v,\Omega)$.  \\

In our arguments below, we require a special test function related to a given weak solution of \eqref{eqn:dirichlet}.  To avoid ambiguity we include the following result taken from \cite{MR4280269} and we refer the reader there for its proof.

\begin{lemma}\label{lem:positivepart} Let $v\in L^1(\Omega)$ and $Q$ satisfy \eqref{eqn:upperellipticity}.  If $r>0$ and $(u,{\bf g})\in QH^1_0(v,\Omega)$, then each of \[\left( \left(u-r\right)_+,{\bf g}\chi_{u>r}\right)~\textrm{and }\left( \left(u-r\right)_-,{\bf g}\chi_{u<r}\right)\]
are elements of $QH^1_0(v,\Omega)$.
\end{lemma}

Key to our proofs is the application of a Sobolev estimate which we include in our hypotheses.  That such estimates hold in general is a difficult question, but there are many examples where they are true allowing for a broad application of our results. The form of these inequalities is dependent on the degeneracy of our matrix $Q$ or how badly its ellipticity fails. In particular, we present a more general notion of gain in the scale of Orlicz spaces.  

\begin{defn}\label{def:sobolev-gain-A} Given a Young function $A(t)$, we say that a degenerate Sobolev inequality with gain $A(t)$ holds if there is a constant $S(A)>0$ so that
\begin{align}\label{eqn:global-sobolev-orlicz}
\|\varphi\|_{L^A(v,\Omega)} \leq 
S(A) \bigg( \int_\Omega |\sqrt{Q}\grad\varphi|^2 \,dx\bigg)^{\frac{1}{2}}
\end{align}
for every $\varphi\in Lip_0(\Omega)$.
\end{defn}
Here, the $L^A(v,\Omega)$ norm is the $v$-weighted Luxembourg norm,
\[\|\varphi\|_{L^A(v,\Omega)} = \inf\left\{ \lambda>0~:~ \int_\Omega A\left(\frac{|\varphi|}{\lambda}\right)~vdx \leq 1\right\}.\]
The form of $A(t)$ depends on the eigenvalues of $Q(x)$.  For example, if $Q(x)$ has lowest eigenvalue that vanishes like a polynomial at a point $x_0\in \Omega$ then $A(t)$ may be chosen as $A(t)=t^{2\sigma}$ for some $\sigma>1$.  In this case, inequality \eqref{eqn:global-sobolev-orlicz} takes the form
\begin{equation} \label{eqn:global-sobolev} 
\bigg( \int_\Omega |\varphi|^{2\sigma } \,\,vdx\bigg)^{\frac{1}{2\sigma }}
\leq
S(\sigma) \bigg( \int_\Omega |\sqrt{Q}\grad\varphi|^2 \,dx\bigg)^{\frac{1}{2}}.
\end{equation}
If the lower eigenvalue vanishes to infinite order, we may expect $A(t) = t^2\log(e+t)^q$ for some $q>0$.  Finally, we say that a degenerate Sobolev inequality without gain holds if \eqref{eqn:global-sobolev-orlicz} is true with $A(t)=t^2$.

\section{Degenerate elliptic equations}
\label{section:pde}

Since this note concerns properties of weak solutions to equations already described in \cite{Tubitak1}, our descriptions here will be brief and the reader is directed to that reference for a fuller discussion. \\

With $Q$ as described in the last section, we begin with first order differential terms.  Fix a measurable ${\bf v}:\Omega\rightarrow \mathbb{R}^n$, ${\bf v} = (v_1,...,v_n)$.  The vector field $V(x) = {\bf v}(x)\cdot \nabla = \sum_{i=1}^n v_i\frac{\partial}{\partial x_i}$ is called a \emph{degenerate subunit vector field} with respect to the weight $v$ and matrix $Q$ if there is a constant $C(V)>0$ so that
\begin{align}\label{eqn:subunit-operator}
\|V\varphi\|_{L^2(v,\Omega)} \leq C(V)\|\varphi\|_{QH^1(v,\Omega)}
\end{align}
for every $\varphi\in Lip_0(\Omega)$.  This condition is easily obtained from a stronger point-wise inequality used in \cite{MR2204824} that requires
\begin{align}\label{eqn:subunit-pointwise}\left({\bf v}(x)\cdot\xi\right)^2 \leq v^{-1}\langle Q(x)\xi,\xi\rangle=v^{-1}|\sqrt{Q(x)}\xi|^2
\end{align}
hold for every $x\in\Omega$. See \cite{Tubitak1} for a wider discussion.  The adjoint of such a vector field will help us to efficiently write our PDE.  Given a subunit vector field $V$ and $\varphi\in Lip_0(\Omega)$, we denote its formal adjoint by 
\[V'\varphi = -\frac{1}{v}\textrm{Div}\left({\bf v}\varphi v\right).\]

Let $N\in\mathbb{N}$ and let ${\bf S}=(S_1,...,S_N), ~{\bf R}=(R_1,...,R_N)$ be $N$-tuples of subunit vector fields $S_j,R_j$ identified with the $\mathbb{R}^n$ valued vector functions ${\bf s}_j, ~{\bf r}_j$.  Fix also measurable functions ${\bf H},~{\bf G}\in L^2(v,\Omega,\mathbb{R}^n)$ and $F\in L^2(v,\Omega)$.  We formally define for $\varphi\in Lip_0(\Omega)$,
\[L\varphi = -\frac{1}{v}\textrm{Div} \left(Q\nabla \varphi\right) + {\bf H}{\bf R}\varphi + {\bf S'G}\varphi + F\varphi\]
where 
\[{\bf HR}\varphi = {\bf H}\cdot (R_1\varphi,...,R_N\varphi)=\sum_{i=1}^N H_iR_i\varphi\]
and
\[{\bf S'G}\varphi = (S_1',...,S_N')\cdot (G_1\varphi,...,G_N\varphi) = \sum_{i=1}^N S_i'(G_i\varphi).\]

As part of the proofs of our main theorems, we require a product rule for $QH^1_0(v,\Omega)$ functions.  The following lemma was proved in \cite{Tubitak1};  we include it here for completeness and refer the reader there for the proof.

\begin{lemma}[Lemma 4.3 of \cite{Tubitak1}] \label{lem:productrule}
    Let $ (u, \nabla u), (w, \nabla w) \in  QH^{1,2}_0(v, \Omega)$ and let $S=\textbf{s}\cdot\nabla$ be a degenerate subunit vector field that satisfies~\eqref{eqn:subunit-operator}. Then each of the following hold:
    \begin{enumerate}
    \item If the global Sobolev inequality \eqref{eqn:global-sobolev} holds with gain $\sigma>1$, then  there exists a unique $L^{\frac{2 \sigma}{\sigma + 1}}(v, \Omega)$ function $S(uw)$ such that $S(uw) = uS(w) + wS(u)$ (with equality in $L^{\frac{2 \sigma}{\sigma + 1}}(v, \Omega)$).
    \item If the global Sobolev inequality \eqref{eqn:global-sobolev-orlicz} holds with gain $A(t)=t^2\log(e+t)^\sigma$, then there exists a unique $L^B(v, \Omega)$ function $S(uw)$ such that $S(uw) = uS(w) + wS(u)$ in $L^B(v, \Omega)$, where  $B(t) = t\log(e+t)^\frac{\sigma}{2}$. 
    \item If the global Sobolev inequality \eqref{eqn:global-sobolev} holds with no gain (i.e., $\sigma=1$ in \eqref{eqn:global-sobolev}), then there exists a unique $L^1(v, \Omega)$ function $S(uw)$ such that $S(uw) = uS(w) + wS(u)$ in $L^1(v, \Omega)$.
    \end{enumerate}
\end{lemma}

\begin{defn} A pair $(u,\nabla u)\in QH^1_0(v,\Omega)$ is called a degenerate weak solution (subsolution,supersolution) of the Dirichlet problem 
\begin{equation}\label{eqn:dirichletwith0}
\begin{cases}
Lu = (\leq,\geq )0 & x \in \Omega, \\
 u = 0, & x \in \partial \Omega,
 \end{cases}
 \end{equation}
 if for every $\varphi\in Lip_0(\Omega)$,
 \begin{align}\label{eqn:weaksolution}
 \int_\Omega \nabla\varphi \cdot Q\nabla u~dx + \int_\Omega \varphi{\bf H}\cdot{\bf R}u~vdx+\int_\Omega u{\bf G}\cdot{\bf S}\varphi~vdx + \int_\Omega F\varphi u~vdx = (\leq,\geq) 0.
 \end{align}
 
\end{defn}
\begin{remark} We note that with our  assumptions, each integral in definition \eqref{eqn:weaksolution} is finite.  
\end{remark}
The existence of degenerate weak solutions satisfying \eqref{eqn:weaksolution} is studied in \cite{Tubitak1} and we refer the reader there for a detailed analysis based on the form of the gain in the Sobolev inequality~\eqref{eqn:global-sobolev-orlicz}.

\section{Statement and proof of the weak maximum principle}
\label{section:main-thm}

In this section we state and prove our main result.  First, we summarize our hypotheses.

\begin{hyp} \label{hyp:global-hyp}
    We make the following assumptions:  
    \begin{itemize}
        \item $\Omega \subset \R^n$ is a domain; a bounded, connected, open set. 

        \item  $v$ is a weight with $v\in L^1(\Omega)$, and $Q$ is a matrix function  satisfying \eqref{eqn:upperellipticity}.
        
        \item $F \in L^2(v,\Omega)$, and $\bG,\, \bH \in L^2(v,\Omega,\R^N)$ for some $N\in \N$.

        \item $\bR,\, \bS$ are $N$-tuples of degenerate subunit vector fields that satisfy \eqref{eqn:subunit-operator}.
        %

        \item $F,\, \bG, \bS$ satisfy the compatibility condition
        \[ \int_\Omega (\bG\cdot\bS \varphi +F \varphi)\,\,vdx\geq 0, \]
        for every non-negative $\varphi \in \lip_0(\Omega)$.  
    \end{itemize}
\end{hyp}
\begin{remark} The compatibility condition is present even in the classical elliptic theory.  See \cite[Chapter 8, (8.8)]{MR1814364}. In \cite[Lemma~4.7]{Tubitak1} it was shown that with our hypotheses, this condition holds with $\varphi=uw\geq 0$, where $u,\,w\in QH_0^{1,2}(v,\Omega)$. 
\end{remark}
\begin{theorem} \label{thm:weak-max}
    Given  Hypothesis~\ref{hyp:global-hyp}, let $u \in QH^{1,2}(v,\Omega)$ satisfy $Lu\leq 0$ in $\Omega$.  Then 
    \begin{equation} \label{eqn:weak-max1}
        \sup_{\Omega} u(x) \leq \sup_{\partial \Omega} u^+(x)
    \end{equation}
 if one of the following is true:
\begin{enumerate}
    \item The degenerate Sobolev inequality~\eqref{eqn:global-sobolev} holds for some $\sigma>1$, and for some $q>2\sigma'$, $\bG,\,\bH \in L^q(v,\Omega,\R^N)$, $F\in L^{\frac{q}{2}}(v,\Omega)$.

     \item The degenerate Sobolev inequality~\eqref{eqn:global-sobolev-orlicz} holds with $A(t)=t^2\log(e+t)^{\sigma}$ for some $\sigma>0$, and  $\bG,\,\bH \in L^C(v,\Omega,\R^N)$ and $F\in L^{D}(v,\Omega)$, where $C(t)=\exp\big(t^{\frac{2}{\lambda\sigma}}\big)-1$ and $D(t)=\exp\big(t^{\frac{1}{\lambda\sigma}}\big)-1$ for some $0<\lambda<1$. 

      \item The degenerate Sobolev inequality~\eqref{eqn:global-sobolev} holds without gain (i.e., $\sigma=1$),  $F\in L^{\infty}(v,\Omega)$,  $\bG,\,\bH \in L^\infty(v,\Omega,\R^N)$, and
\begin{equation} \label{eqn:perturbation}
 \max(C(\bR),C(\bS))S(2,1))\big( \|\bG\|_{L^\infty(v,\Omega)} + \|\bH\|_{L^\infty(v,\Omega)}\big)<1.
\end{equation}

\end{enumerate}
With the same hypotheses, if $Lu \geq 0$, then
    \begin{equation} \label{eqn:weak-max2}
        \inf_{\Omega} u(x) \geq \inf_{\partial \Omega} u^-(x)
    \end{equation}
holds.
\end{theorem}

\begin{proof}
    We will prove~\eqref{eqn:weak-max1}.  The proof of~\eqref{eqn:weak-max2} follows by replacing $u$ by $-u$.    We give a proof by contradiction.  Suppose to the contrary that $\sup_{\Omega} u(x) > \sup_{\partial \Omega} u^+(x)$; then fix any $r>0$ such that
    \[  \sup_{\partial \Omega} u^+(x) <r <  \sup_{\Omega} u(x). \]
    Then by definition, we must have that 
    \[ \varphi = (u-r)^+ = (u^+-r)^+ \in QH_0^{1,2}(v,\Omega). \]
    Therefore, we can use $\varphi$ as a test function in the definition of a subsolution; since $Lu\leq 0$,  we have that 
    \[ \int_\Omega Q\nabla u\cdot \nabla \varphi\, dx + \int_\Omega \bH\cdot\bR u \varphi\, vdx
    + \int_\Omega u\bG \cdot\bS \varphi\, vdx+\int_\Omega Fu\varphi\, vdx \leq 0. \]
In each of cases (1)--(3), using Lemma~\ref{lem:productrule}  we have that $\bG\cdot\bS(u\varphi)=u\bG\cdot\bS \varphi + \varphi\bG\cdot\bS u$.  Moreover, if we let $\Gamma=\supp(\varphi)$, then $\supp(\grad \varphi)\subset \Gamma$ and $\grad u = \grad \varphi$ on $\Gamma$. Thus,
\begin{align*} 
\int_\Gamma Q\nabla \varphi \cdot \nabla \varphi\, dx
& = \int_\Omega Q\nabla u \cdot \nabla \varphi\, dx \\
& \leq \int_\Omega \varphi \big(\bG\cdot\bS u -\bH\cdot\bR u\big)\, vdx
- \int_\Omega \big(\bG\cdot\bS (u\varphi) +Fu\varphi\big)\, vdx.   \\
\intertext{By the compatibility condition in our hypotheses,  again in every case we have that the last integral is nonnegative, so}
& \leq \int_\Omega \varphi \big(\bG\cdot\bS u -\bH\cdot\bR u\big)\, vdx. 
\end{align*}
 Therefore, by H\"older's inequality and using that our subunit vector fields satisfy~\eqref{eqn:subunit-operator},
\begin{align}
 \qquad   & \int_\Gamma Q\nabla \varphi \cdot \nabla \varphi\, dx
    \leq\int_\Gamma \varphi(\bG\cdot\bS-\bH\cdot\bR) \varphi\, vdx \notag\\
    &  \qquad \leq\int_\Gamma|\varphi|\big(|\bG||\bS \varphi|+|\bH||\bR \varphi|\big)\,vdx; \notag \\
    &  \qquad\leq  \|\varphi|\bG|\|_{L^2(v,\Gamma)}\|\bS \varphi\|_{L^2(v,\Gamma)} + \|\varphi|\bH|\|_{L^2(v,\Gamma)}\|\bR \varphi\|_{L^2(v,\Gamma)} \notag \\
    &  \qquad\leq  \bigg(C(\bS)\|\varphi|\bG|\|_{L^2(v,\Gamma)}+C(\bR)\|\varphi|\bH|\|_{L^2(v,\Gamma)} \bigg)\|\varphi\|_{QH^{1,2}(v,\Gamma)}; \notag
    \intertext{by the Sobolev inequality with no gain (which holds in every case),}
    &  \qquad\leq C\big( \bR, \bS, S(2,1)\big)\bigg(\|\varphi|\bG|\|_{L^2(v,\Gamma)}+\|\varphi|\bH|\|_{L^2(v,\Gamma)}\bigg)\|\nabla \varphi\|_{QL^2(v,\Gamma)}.\label{eq:3.5}
    \end{align}

The remainder of the proof consists of three cases, corresponding to assumptions (1)--(3).  Each case is very similar to the proof of the corresponding case in~\cite[Theorem~5.1]{Tubitak1}, so we will only give the first case and refer the reader to this paper for additional details. 
To estimate the first term on the righthand side of~\eqref{eq:3.5}.  By H\"older's inequality applied twice with exponent $\frac{q}{2}$ and then $p=\frac{\sigma(q-2)}{q}>1$, 
\begin{align*}
    \bigg(\|\varphi|\bG|\|_{L^2(v,\Gamma)}+\|\varphi|\bH|\|_{L^2(v,\Gamma)}\bigg)
    &\leq(\|\bG\|_{L^q(v,\Gamma)}+\|\bH\|_{L^q(v,\Gamma)}) \|\varphi\|_{L^{\frac{2q}{q-2}}(v,\Gamma)}\\
    &=C(\bG,\bH)\|\varphi\|_{L^{\frac{2q}{q-2}}(v,\Gamma)}. \\
    &\leq C(\bG,\bH)v\left(\Gamma\right)^{\frac{1}{2\sigma'}-\frac{1}{q}} \|\varphi\|_{L^{2\sigma}(v,\Gamma)}. \\
   &\leq C(\bG,\bH)S(2,\sigma)v\left(\Gamma\right)^{\frac{1}{2\sigma'}-\frac{1}{q}} \|\nabla \varphi\|_{QL^2(\Gamma)};
\end{align*}
The final inequality follows from  Sobolev inequality \eqref{eqn:global-sobolev}.
Given this, \eqref{eq:3.5} becomes
\begin{equation*}
    \int_\Gamma |\sqrt{Q}\nabla \varphi|^2~dx\leq C v\left(\Gamma\right)^{\frac{1}{2\sigma'}-\frac{1}{q}}\|\nabla \varphi\|^2_{QL^2(\Gamma)},
\end{equation*}
where the constant $C$ is independent of $r$.  Therefore, rearranging terms, since 
 $\frac{1}{2\sigma'}-\frac{1}{q}>0$, we get
     \begin{equation*}
         v\left(\Gamma\right)\geq C^{-\frac{2q\sigma'}{q-2\sigma'}}>0. 
         \end{equation*}
In other words, $\supp(\varphi)$ 
has positive $v$-measure independent of $r$.  But,  since $u\in L^p(v,\Omega)$, we must have that as  $r\to \sup_\Omega u$, $\varphi\rightarrow 0$ $v$-a.e. in $\Omega$, and so $v(\Gamma)\to 0$.  From this contradiction we see that \eqref{eqn:weak-max1} must hold. 
\end{proof}

\section*{Conflict of interest}

On behalf of all authors, the corresponding author states that there is no conflict of interest.

\section*{Data availability}

This project has no associated data.

\bibliographystyle{plain}
\bibliography{weak-max}
\end{document}